\documentclass{amsart}
\usepackage{amsfonts}
\usepackage{graphicx}

\newtheorem{theorem}{Theorem}[section]
\newtheorem*{theorem A}{Theorem A}
\newtheorem*{theorem B}{N\"olker's Theorem}

\theoremstyle{remark}
\newtheorem{remark}{Remark}[section]
\theoremstyle{remark}

\theoremstyle{definition}

\newtheorem{example}{Example}[section]
\numberwithin{equation}{section}
\def\({\left ( }
\def\){\right )}
\def\<{\left < }
\def\>{\right >}

\input{tcilatex}

\begin{document}
\title{ Affine translation surfaces in the isotropic 3-space}
\author{Muhittin Evren Aydin$^*$, Mahmut Ergut}
\address{$^*$ Department of Mathematics, Faculty of Science, Firat University, Elazig, 23200, Turkey}
\address{Department of Mathematics, Faculty of Science and Art, Namik Kemal University, Tekirdag, 59000, Turkey}
\email{meaydin@firat.edu.tr, mergut@nku.edu.tr}
\thanks{}

\begin{abstract}
The isotropic 3-space $\mathbb{I}^{3}$ is a real affine
3-space endowed with the metric $dx^{2}+dy^{2}.$ In this paper we describe Weingarten and linear Weingarten
affine translation surfaces in $\mathbb{I}^{3}$. Further we classify the affine translation
surfaces in $\mathbb{I}^{3}$ that satisfy certain equations in terms of the position vector and the Laplace operator.
\end{abstract}

\maketitle

\section{Introduction}

It is well-known that a surface is called \textit{translation surface }in a
Euclidean 3-space $\mathbb{R}^{3}$ if it is the graph of a function $z\left(
x,y\right) =f\left( x\right) +g\left( y\right) $ for the standart coordinate
system of $\mathbb{R}^{3}.$ One of the famous minimal surfaces of $\mathbb{R}%
^{3}$ is Scherk's minimal translation surface which is the graph of the function
\begin{equation*}
z\left( x,y\right) =\frac{1}{c}\log \left\vert \frac{\cos \left( cx\right) }{%
\cos \left( cy\right) }\right\vert ,\text{ }c\in 
%TCIMACRO{\U{211d} }%
%BeginExpansion
\mathbb{R}
%EndExpansion
^{\ast }:=\mathbb{R-}\left\{ 0\right\} .
\end{equation*}%
In order for more generalizations of the translation surfaces to see in various ambient spaces we refer to \cite{4,5,7,12,16,19,20,24,26}.

In 2013, H. Liu and Y. Yu \cite{14} defined the \textit{affine translation
surfaces }in $\mathbb{R}^{3}$ as the graph of the function 
\begin{equation*}
z\left( x,y\right) =f\left( x\right) +g\left( y+ax\right) ,\text{ }a\in 
%TCIMACRO{\U{211d} }%
%BeginExpansion
\mathbb{R}
%EndExpansion
^{\ast }
\end{equation*}%
and described the minimal affine translation surfaces which are given by%
\begin{equation*}
z\left( x,y\right) =\frac{1}{c}\log \left\vert \frac{\cos \left( c\sqrt{%
1+a^{2}}x\right) }{\cos \left( c\left[ y+ax\right] \right) }\right\vert ,%
\text{ }a,c\in 
%TCIMACRO{\U{211d} }%
%BeginExpansion
\mathbb{R}
%EndExpansion
^{\ast }.
\end{equation*}%
These are called \textit{affine Scherk surface}. Then H. Liu and S.D. Jung 
\cite{15} classified the affine translation surfaces in $\mathbb{R}^{3}$ of arbitrary constant mean
curvature.

In the isotropic 3-space $\mathbb{I}^{3}$, there exist three different classes of
translation surfaces given by (see  \cite{18,25})%
\[
\left\{ 
\begin{array}{l}
z\left( x,y\right) =f\left( x\right) +g\left( y\right) , \\ 
y\left( x,z\right) =f\left( x\right) +g\left( z\right) , \\ 
x\left( y,z\right) =\frac{1}{2}\left( f\left( \frac{y+z-\pi }{2}\right)
+g\left( \frac{\pi -y+z}{2}\right) \right),%
\end{array}%
\right. 
\]%
where $x,y,z$ are the standart affine coordinates in $\mathbb{I}^{3}$. These surfaces are respectively called \textit{translation surfaces of Type 1,2,3} in $\mathbb{I}^{3}$. Such surfaces of constant isotropic
Gaussian and mean curvature were obtained in \cite{18} as well as Weingarten
ones.

The translation surfaces of Type 1 in $\mathbb{I}^{3}$ that satisfy the condition%
\begin{equation*}
\bigtriangleup ^{I,II}r_{i}=\lambda _{i}r_{i},\text{ }\lambda _{i}\in 
%TCIMACRO{\U{211d} }%
%BeginExpansion
\mathbb{R}
%EndExpansion
,\text{ }i=1,2,3,
\end{equation*}%
were presented in \cite{13}, where $r_{i}$ is the coordinate function of
the position vector and $\bigtriangleup ^{I,II}$ the Laplace operator
with respect to the first and second fundamental forms, respectively. This condition is natural, being related to the
so-called \textit{submanifolds of finite type}, introduced by B.-Y. Chen in
the late 1970's (see \cite{8,9,11}). More details of translation surfaces in the isotropic spaces can be
found in \cite{2,3,6}.

In this paper, we investigate the affine translation surfaces of Type 1 in $\mathbb{I}%
^{3}$, i.e. the graphs of the function%
\begin{equation*}
z\left( x,y\right) =f\left( ax+by\right) +g\left( cx+dy\right) ,\text{ }%
ad-bc\neq 0
\end{equation*}%
and classify ones of Weinagarten type. Morever we describe the affine
translation surfaces of Type 1 that satisfy the condition $\bigtriangleup
^{I,II}r_{i}=\lambda _{i}r_{i}$.

\section{Preliminaries}

The isotropic 3-space $\mathbb{I}^{3}$ is a real affine space defined from the projective 3-space $P\left( 
%TCIMACRO{\U{211d} }%
%BeginExpansion
\mathbb{R}
%EndExpansion
^{3}\right) $ with an absolute figure consisting of a plane $\omega$ and two complex-conjugate straight lines $f_{1},f_{2}$  in $%
\omega $ (see \cite{1,10,17}, \cite{21}-\cite{23}). Denote the projective coordinates by $\left(X_{0},X_{1},X_{2},X_{3}\right)$
 in $P\left( 
%TCIMACRO{\U{211d} }%
%BeginExpansion
\mathbb{R}
%EndExpansion
^{3}\right) $. Then the absolute plane $\omega $
is given by $X_{0}=0$ and the absolute lines $f_{1},f_{2}$ by $%
X_{0}=X_{1}+iX_{2}=0,$ $X_{0}=X_{1}-iX_{2}=0.$ The intersection point $%
F(0:0:0:1)$ of these two lines is called the absolute point. The group of
motions of $\mathbb{I}^{3}$ is a six-parameter group given in the affine
coordinates $x=\frac{X_{1}}{X_{0}},$ $y=\frac{X_{2}}{X_{0}},$ $z=%
\frac{X_{3}}{X_{0}}$ by

\begin{equation*}
\left( x,y,z\right) \longmapsto \left( x^{\prime },y^{\prime },z^{\prime
}\right) :\left\{ 
\begin{array}{l}
x^{\prime }=a_{1}+x\cos \phi -y\sin \phi , \\ 
y^{\prime }=a_{2}+x\sin \phi +y\cos \phi , \\ 
z^{\prime }=a_{3}+a_{4}x+a_{5}y+z,%
\end{array}%
\right.
\end{equation*}%
where $a_{1},...,a_{5},\phi \in 
%TCIMACRO{\U{211d} }%
%BeginExpansion
\mathbb{R}
%EndExpansion
.$

The metric of $\mathbb{I}^{3}$ is induced by the absolute figure, i.e. $ds^{2}=dx^{2}+dy^{2}.$ The lines in $%
z-$direction are called \textit{isotropic lines}. The planes
containing an isotropic line are called \textit{isotropic planes}. Other
planes are \textit{non-isotropic}.

Let $M$ be a surface immersed in $\mathbb{I}^{3}$. We call the surface $M$ \textit{admissible} if it has no
isotropic tangent planes. Such a surface can get the form%
\begin{equation}
r:D\subseteq \mathbb{R}^{2}\longrightarrow \mathbb{I}^{3}:\text{ }\left(
x,y\right) \longmapsto \left( r_{1}\left( x,y\right)
,r_{2}\left( x,y\right) ,r_{3}\left( x,y\right) \right). 
\notag
\end{equation}%

The components $E,F,G$ of the first fundamental form $I$ of $M$ can be calculated via the metric induced from $\mathbb{I}^{3}$. 

Denote $\bigtriangleup^I$ the Laplace operator of $M$ with respect to $I$. Then it is
defined as 
\begin{equation}
\bigtriangleup \phi =\frac{1}{\sqrt{\left\vert W\right\vert }}\left\{ 
\frac{\partial }{\partial x}\left( \frac{G\phi _{x}-F\phi _{y}}{\sqrt{%
\left\vert W\right\vert }}\right) -\frac{\partial }{\partial y}\left( \frac{%
F\phi _{x}-E\phi _{y}}{\sqrt{\left\vert W\right\vert }}\right) \right\},
\tag{2.1}
\end{equation}
where $\phi $ is a smooth function on $M$ and $W=EG-F^2$.

The unit normal vector field of $M^{2}$ is completely isotropic, i.e. $%
\left( 0,0,1\right) $. Morever, the components of the second fundamental form $II$
are%
\begin{equation}
L=\frac{\det \left( r_{xx},r_{x},r_{y}\right) }{\sqrt{%
EG-F^2}},\text{ }M=\frac{\det \left( r_{xy},r_{x},r_{y}\right) }{\sqrt{%
EG-F^2}},\text{ }N=\frac{\det \left( r_{yy},r_{x},r_{y}\right) }{\sqrt{%
EG-F^2 }},  \tag{2.2}
\end{equation}%
where $r_{xy}=\frac{\partial ^{2}r}{\partial x\partial y},$ etc.

The \textit{relative curvature} (so-called the \textit{isotropic
curvature} or \textit{isotropic Gaussian curvature}) and the \textit{%
isotropic mean curvature} are respectively defined by%
\begin{equation}
K=\frac{LN-M^2 }{EG-F^2},\text{ }H=%
\frac{EN-2FM+LG}{2(EG-F^2) }.
\tag{2.3}
\end{equation}%

Assume that nowhere $M$ has parabolic points, i.e. $K\neq 0.$ Then the Laplace
operator with respect to $II$ is given by%
\begin{equation}
\bigtriangleup ^{II}\phi =-\frac{1}{\sqrt{\left\vert w\right\vert }}\left\{ 
\frac{\partial }{\partial x}\left( \frac{N\phi _{x}-M\phi _{y}}{\sqrt{%
\left\vert w\right\vert }}\right) -\frac{\partial }{\partial y}\left( \frac{%
M\phi_{x}-L\phi _{y}}{\sqrt{\left\vert w\right\vert }}\right) \right\} 
\tag{2.4}
\end{equation}%
for a smooth function $\phi $ on $M$ and $w=LN-M^2$.

In particular, if $M$ is a graph surface in $\mathbb{I}^{3}$ of a smooth function $z(x,y)$ then the metric on $M$ induced from $\mathbb{I}^{3}$ is given by $%
dx^{2}+dy^{2}.$ Thus its Laplacian turns to%
\begin{equation}
\bigtriangleup ^{I}=\frac{\partial ^{2}}{\partial x^{2}}+\frac{\partial ^{2}%
}{\partial y^{2}}.  \tag{2.5}
\end{equation}

The matrix of second fundamental form $II$ of $M$ corresponds to the Hessian
matrix $\mathcal{H}\left( z\right) $,\ i.e., 
\begin{equation*}
\begin{pmatrix}
L & M \\ 
M & N%
\end{pmatrix}%
=%
\begin{pmatrix}
z_{xx} & z_{xy} \\ 
z_{xy} & z_{yy}%
\end{pmatrix}%
.
\end{equation*}%
Accordingly, the formulas (2.3) reduce to%
\begin{equation}
K=\det \left( \mathcal{H}\left( z\right) \right) ,\text{ }H=\frac{%
trace\left( \mathcal{H}\left( z\right) \right) }{2}.  \tag{2.6}
\end{equation}

\section{Weingarten affine translation surfaces}

Let $M$ be the graph surface in $\mathbb{I}^{3}$ of the function $z\left( x,y\right) =f\left( u\right) +g\left( v\right)$, where
\begin{equation}
u=ax+by, \text{  } v=cx+dy.  \tag{3.1}
\end{equation}%
If $ad-bc\neq 0$, we call the surface $M$ \textit{affine translation surface of Type 1} in $\mathbb{I}^{3}$ and the pair $\left( u,v\right) $
 \textit{affine parameter coordinates}. 

In the particular case $a=d=1$ and $b=c=0$ (or $a=d=0$ and $b=c=1$), such a surface reduces to
the translation surface of Type 1 in $\mathbb{I}^{3}$. Let us fix some notations to use remaining part:
\begin{equation*}
\frac{\partial f}{\partial x}=a\frac{df}{du}=af^{\prime },\text{ }\frac{%
\partial f}{\partial y}=bf^{\prime },\text{ }\frac{\partial g}{\partial x}=c%
\frac{dg}{dv}=cg^{\prime },\text{ }\frac{\partial g}{\partial y}=dg^{\prime
},
\end{equation*}
and so on. By $\left( 2.6\right) ,$ the relative curvature $K$ and the isotropic mean
curvature $H$ of $M$ turn to%
\begin{equation}
K=\left( ad-bc\right) ^{2}f^{\prime \prime }g^{\prime \prime }\text{ and }%
2H=\left( a^{2}+b^{2}\right) f^{\prime \prime }+\left( c^{2}+d^{2}\right)
g^{\prime \prime }.  \tag{3.2}
\end{equation}%

Now we can state the following result to describe the Weingarten affine
translation surfaces of Type 1 in $\mathbb{I}^{3}$ that satisfy the condition
\begin{equation}
K_{x}H_{y}-K_{y}H_{x}=0,  \tag{3.3}
\end{equation}
where the subscript denotes the partial derivative.

\begin{theorem} Let $M$ be a Weingarten affine translation surface of Type 1 in $\mathbb{I}%
^{3}.$ Then one of the following occurs:

(i) $M$ is a quadric surface given by%
\begin{equation*}
z\left( x,y\right) =c_{1}u^{2}+\frac{c_{1}\left(
a^{2}+b^{2}\right) }{\left( c^{2}+d^{2}\right) }v^{2}+c_{2}u+c_{3}v+c_{4},\text{ }c_{1},...,c_{4}\in 
%TCIMACRO{\U{211d} }%
%BeginExpansion
\mathbb{R}
%EndExpansion
;
\end{equation*}

(ii)$\ M$ is of the form either 
\begin{equation*}
z\left( x,y\right) =f\left( u\right) +c_{1}v^{2}+c_{2}v+c_{3},\text{ 
} f^{\prime \prime \prime} \neq 0,  \text{ 
}c_{1},c_{2},c_{3}\in 
%TCIMACRO{\U{211d}}%
%BeginExpansion
\mathbb{R}%
%EndExpansion
\end{equation*}%
or%
\begin{equation*}
z\left( x,y\right) =g\left( v\right) +c_{1}u^{2}+c_{2}u+c_{3}, \text{ 
} g^{\prime \prime \prime} \neq 0, \text{ 
}c_{1},c_{2},c _{3}\in 
%TCIMACRO{\U{211d} }%
%BeginExpansion
\mathbb{R}
%EndExpansion
,
\end{equation*}%
where $\left( u,v\right) $ is the affine parameter coordinates given by (3.1).
\end{theorem}

\begin{remark}
We point out that a \textit{quadric surface} in $\mathbb{I}^{3}$ is the set of the points satisfying an equation of the second degree.
\end{remark}

\begin{proof} It follows from $\left( 3.2\right) $ and $\left( 3.3\right) $ that

\begin{equation}
\left[ \left( a^{2}+b^{2}\right) f^{\prime \prime }-\left(
c^{2}+d^{2}\right) g^{\prime \prime }\right] f^{\prime \prime \prime
}g^{\prime \prime \prime }=0.  \tag{3.4}
\end{equation}%
To solve $\left( 3.4\right),$ we have several cases:

\textbf{Case (a). } $\left( a^{2}+b^{2}\right) f^{\prime \prime }=\left( c^{2}+d^{2}\right)g^{\prime \prime }$. Then we derive
\begin{equation}
z\left( x,y\right) =c_{1}u^{2}+\frac{c_{1}\left(
a^{2}+b^{2}\right) }{\left( c^{2}+d^{2}\right) }v^{2}+c_{2}u+c
_{3}v+c_{4},\text{ }c_{1},...,c _{4}\in 
%TCIMACRO{\U{211d} }%
%BeginExpansion
\mathbb{R}
%EndExpansion
, \notag
\end{equation}%
which gives the statement (i) of the theorem.

\textbf{Case (b). }$\left( a^{2}+b^{2}\right) f^{\prime \prime } \neq \left( c^{2}+d^{2}\right)g^{\prime \prime }$.  Then, by (3.4), the surface has the form either %
\begin{equation}
z\left( x,y\right) =g\left( v\right) +c_{1}u^{2}+c_{2}u+c_{3},\text{ 
} g^{\prime \prime \prime} \neq 0
\notag
\end{equation}%
or%
\begin{equation}
z\left( x,y\right) =f\left( u\right) +c_{4}v^{2}+c_{5}v+c_{6},\text{ 
} f^{\prime \prime \prime} \neq 0, \text{ 
} c_{1},...,c_{6}\in 
%TCIMACRO{\U{211d} }%
%BeginExpansion
\mathbb{R}
%EndExpansion
. \notag
\end{equation}%
This implies the second statement of the theorem.
Therefore the proof is completed.
\end{proof}

Now we intend to find the linear Weingarten affine translation surfaces of Type 1
in $\mathbb{I}^{3}$ that satisfy%
\begin{equation}
\alpha K+\beta H=\gamma ,\text{ } \alpha ,\beta ,\gamma \in \mathbb{R},\text{ } \left( \alpha ,\beta ,\gamma \right) \neq
\left( 0,0,0\right) .  \tag{3.5}
\end{equation}%
Without lose of generality, we may assume $\alpha \neq 0$ in $\left(
3.5\right) $ and thus it can be rewritten as
\begin{equation}
K+2m_{0}H=n_{0},\text{ }2m_{0}=\frac{\beta }{\alpha },\text{ }n_{0}=\frac{%
\gamma }{\alpha }.  \tag{3.6}
\end{equation}%
Hence the following result can be given.

\begin{theorem} Let $M$ be a linear Weingarten affine translation surface of Type 1 in $%
\mathbb{I}^{3}$ that satisfies $\left( 3.6\right) $. Then we have:

(i) $M$ is a quadric surface given by 
\begin{equation*}
z\left( x,y\right) =c_{1}u^{2}+c_{2}v^{2}+c
_{3}u+c_{4}v+c _{5},\text{ }c_{1},...,c_{5}\in 
%TCIMACRO{\U{211d} }%
%BeginExpansion
\mathbb{R}
%EndExpansion
;
\end{equation*}

(ii) $M$ is of the form either 
\begin{equation*}
z\left( x,y\right) =f\left( u\right) -\frac{m_{0}\left( a^{2}+b^{2}\right) }{%
2\left( ad-bc\right) ^{2}}v^{2}+c_{1}v+c_{2}, \text{ 
} f^{\prime \prime \prime} \neq 0,\text{ }c_{1},c_{2}\in 
%TCIMACRO{\U{211d}}%
%BeginExpansion
\mathbb{R}%
%EndExpansion
\end{equation*}%
or%
\begin{equation*}
z\left( x,y\right) =g\left( v\right) -\frac{m_{0}\left( c^{2}+d^{2}\right) }{%
2\left( ad-bc\right) ^{2}}u^{2}+c_{1}u+c_{2},\text{ 
} g^{\prime \prime \prime} \neq 0, \text{ }c_{1},c_{2}\in 
%TCIMACRO{\U{211d} }%
%BeginExpansion
\mathbb{R}
%EndExpansion
,
\end{equation*}

where $\left( u,v\right) $ is the affine parameter coordinates given by (3.1).
\end{theorem}

\begin{proof} Substituting $\left( 3.2\right) $ in $\left( 3.6\right) $ gives%
\begin{equation}
\left( ad-bc\right) ^{2}f^{\prime \prime }g^{\prime \prime }+m_{0}\left(
a^{2}+b^{2}\right) f^{\prime \prime }+m_{0}\left( c^{2}+d^{2}\right)
g^{\prime \prime }=n_{0}.  \tag{3.7}
\end{equation}%
After taking partial derivative of $\left( 3.7\right) $ with respect to $u$
and $v,$ we deduce $f^{\prime \prime \prime }g^{\prime \prime \prime }=0$.
If both $f^{\prime \prime \prime }$ and $g^{\prime \prime \prime }$ are zero
then we easily obtain the first statement of the theorem. Otherwise, we have
the second statement of the theorem. This proves the theorem.
\end{proof}

\begin{example} Consider the affine translation surface of Type 1 in $\mathbb{I}^{3}$ with%
\begin{equation*}
z\left( x,y\right) =\cos \left( x-y\right) +\left( x+y\right) ^{2},\text{ }-%
\frac{\pi }{6}\leq x,y\leq \frac{\pi }{6}.
\end{equation*}%
This surface plotted as in Fig. 1 satisfies the conditions to be Weingarten and linear Weingarten.
\end{example}

\section{Affine translation surfaces satisfying $\bigtriangleup ^{I,II}r_{i}=%
\protect\lambda _{i}r_{i}$}
This section is devoted to classify the affine translation surfaces of Type 1 in $%
\mathbb{I}^{3}$ that satisfy the conditions $\bigtriangleup ^{I,II}r_{i}=\lambda _{i}r_{i}$%
, $\lambda _{i}\in 
%TCIMACRO{\U{211d} }%
%BeginExpansion
\mathbb{R}
%EndExpansion
.$ For this, we get a local parameterization on such a surface as follows 
\begin{equation}
\left. 
\begin{array}{r}
r\left( x,y\right) =\left( r_{1}\left( x,y\right) ,r_{2}\left( x,y\right)
,r_{3}\left( x,y\right) \right)  \\ 
=\left( x,y,f\left( ax+by\right) +g\left( cx+dy\right) \right) .%
\end{array}%
\right.   \tag{4.1}
\end{equation}%
Thus we first give the following result.

\begin{theorem} Let $M$ be an affine translation surface of Type 1 in $\mathbb{I}%
^{3} $ that satisfies $\bigtriangleup ^{I}r_{i}=\lambda _{i}r_{i}.$ Then it is
congruent to one of the following surfaces:

(i) $\left( \lambda _{1},\lambda _{2},\lambda _{3}\right)=(0,0,0)$
\begin{equation*}
z\left( x,y\right) =c_{1}u^{2}-\frac{c_{1}\left(
a^{2}+b^{2}\right) }{\left( c^{2}+d^{2}\right) }v^{2}+c _{3}u+c_{4}v+c_{5};
\end{equation*}

(ii) $\left( \lambda _{1},\lambda _{2},\lambda _{3}\right)=(0,0,\lambda>0)$
\begin{equation*}
z\left( x,y\right) =c_{1}e^{\sqrt{\frac{\lambda }{a^{2}+b^{2}}}%
u}+c_{2}e^{-\sqrt{\frac{\lambda }{a^{2}+b^{2}}}u}+c_{3}e^{%
\sqrt{\frac{\lambda }{c^{2}+d^{2}}}v}+c_{4}e^{-\sqrt{\frac{\lambda }{%
c^{2}+d^{2}}}v} ;
\end{equation*}

(iii) $\left( \lambda _{1},\lambda _{2},\lambda _{3}\right)=(0,0,\lambda<0)$
\begin{eqnarray*}
z\left( x,y\right) &=&c_{1}\cos \left( \sqrt{\tfrac{-\lambda }{%
a^{2}+b^{2}}}u\right) +c_{2}\sin \left( \sqrt{\tfrac{-\lambda }{%
a^{2}+b^{2}}}u\right) +c_{3}\cos \left( \sqrt{\tfrac{-\lambda }{c^{2}+d^{2}}}v\right)\\
&&+c_{4}\sin \left( \sqrt{\tfrac{-\lambda }{c^{2}+d^{2}}}v\right) ,
\end{eqnarray*}

where $\left( u,v\right) $ is the affine parameter coordinates given by (3.1) and $c
_{1},...,c_{5}\in 
%TCIMACRO{\U{211d} }%
%BeginExpansion
\mathbb{R}
%EndExpansion
$.
\end{theorem}

\begin{proof} It is easy to compute from $\left( 2.5\right) $ and $\left(
4.1\right) $ that%
\begin{equation}
\bigtriangleup ^{I}r_{1}=\bigtriangleup ^{I}r_{2}=0  \tag{4.2}
\end{equation}%
and%
\begin{equation}
\bigtriangleup ^{I}r_{3}=\left( a^{2}+b^{2}\right) f^{\prime \prime }+\left(
c^{2}+d^{2}\right) g^{\prime \prime }.  \tag{4.3}
\end{equation}%
Assuming $\bigtriangleup ^{I}r_{i}=\lambda _{i}r_{i}$, $i=1,2,3$, in $\left( 4.2\right) $
and $\left( 4.3\right) $ yields $\lambda _{1}=\lambda _{2}=0$ and%
\begin{equation}
\left( a^{2}+b^{2}\right) f^{\prime \prime }+\left( c^{2}+d^{2}\right)
g^{\prime \prime }=\lambda \left( f+g\right) ,\text{ }\lambda _{3}=\lambda .
\tag{4.4}
\end{equation}%
If $\lambda =0$ in $\left( 4.4\right) ,$ then we derive%
\begin{equation*}
f\left( u\right) =c_{1}u^{2}+c _{2}u+c _{3}
\end{equation*}%
and 
\begin{equation*}
g\left( v\right) =-\frac{c_{1}\left( a^{2}+b^{2}\right) }{\left(
c^{2}+d^{2}\right) }v^{2}+c _{4}v+c_{5},\text{ }c
_{1},...,c_{5}\in 
%TCIMACRO{\U{211d} }%
%BeginExpansion
\mathbb{R}
%EndExpansion
,
\end{equation*}%
which proves the statement (i) of the theorem.

If $\lambda \neq 0$ then $\left( 4.4\right) $ can be rewritten as%
\begin{equation}
\left( a^{2}+b^{2}\right) f^{\prime \prime }-\lambda f=\mu =-\left(
c^{2}+d^{2}\right) g^{\prime \prime }+\lambda g,\text{ }\mu \in 
%TCIMACRO{\U{211d} }%
%BeginExpansion
\mathbb{R}
%EndExpansion
.  \tag{4.5}
\end{equation}
In the case $\lambda >0,$ by solving $\left( 4.5\right) $ we obtain%
\begin{equation}
\left\{ 
\begin{array}{l}
f\left( u\right) =c_{1}\exp \left( \sqrt{\frac{\lambda }{a^{2}+b^{2}}}%
u\right) +c_{2}\exp \left( -\sqrt{\frac{\lambda }{a^{2}+b^{2}}}%
u\right) +\frac{\mu }{\lambda }, \\ 
g\left( v\right) =c_{3}\exp \left( \sqrt{\frac{\lambda }{c^{2}+d^{2}}}%
v\right) +c_{4}\exp \left( -\sqrt{\frac{\lambda }{c^{2}+d^{2}}}%
v\right) -\frac{\mu }{\lambda },\text{ }%
\end{array}%
\right.  \notag
\end{equation}%
where $c_{1},...,c_{4}\in 
%TCIMACRO{\U{211d} }%
%BeginExpansion
\mathbb{R}
%EndExpansion
.$ This gives the statement (ii) of the theorem.

Otherwise, i.e., $\lambda <0,$ then we derive%
\begin{equation}
\left\{ 
\begin{array}{c}
f\left( u\right) =c_{1}\cos \left( \sqrt{\tfrac{-\lambda }{a^{2}+b^{2}%
}}u\right) +c_{2}\sin \left( \sqrt{\tfrac{-\lambda }{a^{2}+b^{2}}}%
u\right) +\frac{\mu }{\lambda }, \\ 
g\left( v\right) =c_{3}\cos \left( \sqrt{\tfrac{-\lambda }{c^{2}+d^{2}%
}}v\right) +c_{4}\sin \left( \sqrt{\tfrac{-\lambda }{c^{2}+d^{2}}}%
v\right) -\frac{\mu }{\lambda }%
\end{array}%
\right. \notag
\end{equation}%
for $c_{1},...,c_{4}\in 
%TCIMACRO{\U{211d} }%
%BeginExpansion
\mathbb{R}
%EndExpansion
.$ This completes the proof.
\end{proof}

\begin{example} Take the affine translation surface of Type 1 in $\mathbb{I}^{3}$ with%
\begin{equation*}
z\left( x,y\right) =\cos \left( x+y\right) +\sin \left( x-y\right) ,\text{ }%
-\pi \leq x,y\leq \pi.
\end{equation*}%
Then it satisfies $\bigtriangleup ^{I}r_{i}=\lambda _{i}r_{i}$ for $%
\lambda _{1}=\lambda _{2}=0$, $\lambda _{3}=-2$ and can be drawn as in
Fig. 2.%
\end{example}

Next, we consider the affine translation surface of Type 1 in $\mathbb{I}^{3}$
that satisfies $\bigtriangleup ^{II}r_{i}=\lambda _{i}r_{i}$, $\lambda _{i}\in 
%TCIMACRO{\U{211d} }%
%BeginExpansion
\mathbb{R}
%EndExpansion
.$ Then its Laplace operator with respect to the second fundamental form $II$  has the form%
\begin{equation}
\begin{array}{c}
\bigtriangleup ^{II}\phi =\frac{\left( f^{\prime \prime }g^{\prime \prime
}\right) ^{-2}}{2\left( ad-bc\right) }\left[ \left( -b\phi _{x}+a\phi
_{y}\right) \left( f^{\prime \prime }\right) ^{2}g^{\prime \prime \prime
}+\left( d\phi _{x}-c\phi _{y}\right) f^{\prime \prime \prime }\left(
g^{\prime \prime }\right) ^{2}\right] \\ 
+\frac{\left( f^{\prime \prime }g^{\prime \prime }\right) ^{-1}}{\left(
ad-bc\right) ^{2}}\left[ \left( 2ab\phi _{xy}-b^{2}\phi _{xx}-a^{2}\phi
_{yy}\right) f^{\prime \prime }+\left( 2cd\phi _{xy}-d^{2}\phi
_{xx}-c^{2}\phi _{yy}\right)g^{\prime \prime } \right]%
\end{array}
\tag{4.6}
\end{equation}%
for a smooth function $\phi $ and $f^{\prime \prime }g^{\prime \prime
}\neq 0.$ Hence we have the following result.

\begin{theorem} Let $M$ be an affine translation surface of Type 1 in\textit{\ }$%
\mathbb{I}^{3}$ that satisfies\textit{\ }$\bigtriangleup ^{II}r_{i}=\lambda
_{i}r_{i}.$ Then it is congruent to one of the following surfaces:

(i) $\left( \lambda _{1} \neq 0,\lambda _{2} \neq 0,0\right) $
\[
z\left( x,y\right) =\ln \left( x^{\frac{1}{\lambda _{1}}}y^{\frac{1}{\lambda
_{2}}}\right) +c_{1},\text{ } c_{1}\in 
%TCIMACRO{\U{211d} }%
%BeginExpansion
\mathbb{R}
%EndExpansion
;
\]%

(ii) $\left( \lambda \neq 0,\lambda ,0\right) $
\[
z\left( x,y\right) = \ln \left(\left( uv\right)^{\frac{1}{\lambda }}\right)
+c_{1},\text{ } c_{1}\in 
%TCIMACRO{\U{211d} }%
%BeginExpansion
\mathbb{R}
%EndExpansion
, 
\]%
where $\left( u,v\right) $ is the affine parameter coordinates given by $%
\left( 3.1\right) $.
\end{theorem}

\begin{proof} Let us assume\ that $\bigtriangleup ^{II}r_{i}=\lambda _{i}r_{i}$, $%
\lambda _{i}\in 
%TCIMACRO{\U{211d} }%
%BeginExpansion
\mathbb{R}
%EndExpansion
.$ Then, from $\left( 4.1\right) $ and $\left( 4.6\right) ,$ we state the
following system %
\begin{equation}
d\frac{f^{\prime \prime \prime }}{\left( f^{\prime \prime }\right) ^{2}}-b%
\frac{g^{\prime \prime \prime }}{\left( g^{\prime \prime }\right) ^{2}}%
=2\left( ad-bc\right) \lambda _{1}x,  \tag{4.7}
\end{equation}%
\begin{equation}
-c\frac{f^{\prime \prime \prime }}{\left( f^{\prime \prime }\right) ^{2}}+a%
\frac{g^{\prime \prime \prime }}{\left( g^{\prime \prime }\right) ^{2}}%
=2\left( ad-bc\right) \lambda _{2}y,  \tag{4.8}
\end{equation}%
\begin{equation}
\frac{f^{\prime \prime \prime }f^{\prime }}{\left( f^{\prime \prime }\right)
^{2}}+\frac{g^{\prime \prime \prime }g^{\prime }}{\left( g^{\prime \prime
}\right) ^{2}}-4=2\lambda _{3}\left(
f+g\right) .  \tag{4.9}
\end{equation}%
In order to solve above system we have to distinguish two cases depending on the constants $a,b,c,d$ for $ad-bc \neq 0$ . 

\textbf{Case (a).} Two of $a,b,c,d$ are zero. Without loss of generality we may assume that $b=c=0$ and $a=d=1$. Then the equations (4.7) and (4.8) reduce to
\begin{equation}
\frac{f^{\prime \prime \prime }}{\left( f^{\prime \prime }\right) ^{2}}%
=2\lambda _{1}x  \tag{4.10}
\end{equation}%
and%
\begin{equation}
\frac{g^{\prime \prime \prime }}{\left( g^{\prime \prime }\right) ^{2}}%
=2\lambda _{2}y.  \tag{4.11}
\end{equation}
If $\lambda_{1}=\lambda_{2}=0$ then we obtain a contradiction from (4.9) since $f,g$ are non-constant functions. Thereby we need to consider the remaining cases:

\textbf{Case (a.1).} $\lambda _{1}=0,$ i.e. $f^{\prime \prime \prime }=0.$
Then substituting $\left( 4.10\right) $ and $\left( 4.11\right) $ into $%
\left( 4.9\right) $ implies $\lambda _{3}=0$ and 
\[
g\left( y\right) =\frac{2}{\lambda _{2}}\ln y+c_{1},\text{ }c_{1}\in 
%TCIMACRO{\U{211d} }%
%BeginExpansion
\mathbb{R}
%EndExpansion
.
\]%
Substituting it in (4.11) gives a contradiction.

\textbf{Case (a.2).} $\lambda _{2}=0,$ i.e. $g^{\prime \prime \prime }=0.$
Hence we can similarly obtain $\lambda _{3}=0$ and
\[
f\left( x\right) =\frac{2}{\lambda _{1}}\ln x+c_{1},\text{ }c_{1}\in 
%TCIMACRO{\U{211d} }%
%BeginExpansion
\mathbb{R}
%EndExpansion
,
\]%
which gives a contradiction by considering it into (4.10).

\textbf{Case (a.3). }$\lambda _{1}\lambda _{2}\neq 0.$ By substituting $\left( 4.10\right) $ and $\left(
4.11\right) $ into $\left( 4.9\right) $ we deduce
\begin{equation}
\lambda_{1} xf^{\prime }+\lambda_{2} yg^{\prime }-2=\lambda _{3}\left( f+g\right) .  \tag{4.12}
\end{equation}%

\textbf{Case (a.3.1).} If $\lambda _{3}=0$, then (4.12) reduces to
\begin{equation}
\lambda_{1} xf^{\prime }+\lambda_{2} yg^{\prime }=2.  \tag{4.13}
\end{equation}%
By solving $\left( 4.13\right) $ we find%
\[
f\left( x\right) =\frac{\xi }{\lambda_{1} }\ln x+c_{1}\text{ and }g\left(
v\right) =\frac{2-\xi }{\lambda _{2}}\ln y+c_{2},\text{ }c_{1},c_{2}\in 
%TCIMACRO{\U{211d} }%
%BeginExpansion
\mathbb{R}
%EndExpansion
,\text{ }\xi \in 
%TCIMACRO{\U{211d} }%
%BeginExpansion
\mathbb{R}
%EndExpansion
^{\ast }. \tag{4.14}
\]
Substituting (4.14) into (4.10) and (4.11) yields $\xi =1$. 
This proves the first statement of the theorem.

\textbf{Case (a.3.2).} If $\lambda _{3}\neq 0$ in (4.12) then we can rewrite it as
\begin{equation}
\lambda _{1} xf^{\prime }-\lambda _{3}f-2=c_{1}=-\lambda _{2} yg^{\prime }+\lambda
_{3}g,\text{ }c_{1}\in 
%TCIMACRO{\U{211d} }%
%BeginExpansion
\mathbb{R}
%EndExpansion
.  \tag{4.15}
\end{equation}%
After solving $\left( 4.15\right) ,$ we conclude%
\[
f\left( x\right) =-\frac{2+c_{1}}{\lambda _{3}}+c_{2}x^{\frac{\lambda _{3}}{%
\lambda _{}1}}  \tag{4.16}
\]%
and%
\[
g\left( y\right) =\frac{c_{1}}{\lambda _{3}}+c_{3}y^{\frac{\lambda _{3}}{%
\lambda _{2}}}, \text{ } c_{2},c_{3}\in \mathbb{R}.  \tag{4.17}
\]
By considering (4.16) and (4.17) into (4.10) and (4.11), respectively, we conclude $\lambda _{3}=0$, which implies that this case is not possible.

\textbf{Case (b).} At most one of $a,b,c,d$ is zero. Suppose that $\lambda _{1}=0$ in (4.7). It follows
from $\left( 4.7\right) $ that%
\begin{equation}
\frac{f^{\prime \prime \prime }}{\left( f^{\prime \prime }\right) ^{2}}=%
\frac{c_{1}}{d}\text{ and }\frac{g^{\prime \prime \prime }}{\left( g^{\prime
\prime }\right) ^{2}}=\frac{c_{1}}{b},\text{ }c_{1}\in 
%TCIMACRO{\U{211d} }%
%BeginExpansion
\mathbb{R}
%EndExpansion
,  \tag{4.18}
\end{equation}%
where we may assume that $b\neq 0 \neq d$ since at most one of $a,b,c,d$ can vanish. If $c_{1}=0$ then we derive a contradiction from $\left( 4.9\right) $ since $%
f^{\prime \prime }g^{\prime \prime }\neq 0.$ Considering $\left( 4.18\right) 
$ into $\left( 4.8\right) $ yields $\frac{c_{1}}{bd}=2\lambda _{2}y,$ which
is no possible since $y$ is an independent variable. This implies that $\lambda _{1}$ is not zero and it can be similarly shown that $\lambda _{2}$ is not zero. Hence from $\left( 4.7\right) $ and $\left(
4.8\right) $ we can write%
\begin{equation}
\frac{f^{\prime \prime \prime }}{\left( f^{\prime \prime }\right) ^{2}}%
=2\left( \lambda _{1}ax+\lambda _{2}by\right)   \tag{4.19}
\end{equation}%
and%
\begin{equation}
\frac{g^{\prime \prime \prime }}{\left( g^{\prime \prime }\right) ^{2}}%
=2\left( \lambda _{1}cx+\lambda _{2}dy\right) .  \tag{4.20}
\end{equation}%
Compatibility condition in $\left( 4.19\right) 
$ or $\left( 4.20\right) $ gives $\lambda _{1}=\lambda _{2}.$ Put $\lambda
_{1}=\lambda _{2}=\lambda .$ By substituting $\left( 4.19\right) $ and $\left(
4.20\right) $ into $\left( 4.9\right) $ we deduce
\begin{equation}
\lambda uf^{\prime }+\lambda vg^{\prime }-2=\lambda _{3}\left( f+g\right),  \tag{4.21}
\end{equation}%
where $\left( u,v\right) $ is the affine parameter coordinates given by $%
\left( 3.1\right) $.

\textbf{Case (b.1).} If $\lambda _{3}=0$, then (4.21) reduces to
\begin{equation}
\lambda uf^{\prime }+\lambda vg^{\prime }=2.  \tag{4.22}
\end{equation}%
By solving $\left( 4.22\right) $ we find%
\[
f\left( u\right) =\frac{\xi }{\lambda }\ln u+c_{1}\text{ and }g\left(
v\right) =\frac{2-\xi }{\lambda }\ln v+c_{2},\text{ }c_{1},c_{2}\in 
%TCIMACRO{\U{211d} }%
%BeginExpansion
\mathbb{R}
%EndExpansion
,\text{ }\xi \in 
%TCIMACRO{\U{211d} }%
%BeginExpansion
\mathbb{R}
%EndExpansion
^{\ast }. \tag{4.23}
\]
Substituting (4.23) into (4.19) and (4.20) yields $\xi =1$. 
This proves the second statement of the theorem.

\textbf{Case (b.2).} If $\lambda _{3}\neq 0$ in (4.11) then we can rewrite it as
\begin{equation}
\lambda uf^{\prime }-\lambda _{3}f-2=c_{1}=-\lambda vg^{\prime }+\lambda
_{3}g,\text{ }c_{1}\in 
%TCIMACRO{\U{211d} }%
%BeginExpansion
\mathbb{R}
%EndExpansion
.  \tag{4.24}
\end{equation}%
After solving $\left( 4.24\right) ,$ we deduce%
\[
f\left( u\right) =-\frac{2+c_{1}}{\lambda _{3}}+c_{2}u^{\frac{\lambda _{3}}{%
\lambda }}  \tag{4.25}
\]%
and%
\[
g\left( v\right) =\frac{c_{1}}{\lambda _{3}}+c_{3}v^{\frac{\lambda _{3}}{%
\lambda }}, \text{ } c_{2},c_{3}\in \mathbb{R}.  \tag{4.26}
\]
Considering (4.25) and (4.26) into (4.19) and (4.20), respectively, we find $\lambda _{3}=0$, however this is a contradiction.

\end{proof}

\begin{example} Given the affine translation surface of Type 1 in\textbf{\ }$\mathbb{I}^{3}$ as
follows%
\[
z\left( x,y\right) =\ln\left( 2x+y\right) +\ln \left(  x-y\right), (u,v) \in [3,5] \times [1,2] %
.
\]%
Then it holds $\bigtriangleup ^{II}r_{i}=\lambda _{i}r_{i}$ for $\left(
\lambda _{1},\lambda _{2},\lambda _{3}\right) =\left( 1,1,0\right) $ and we
plot it as in Fig. 3.%
\end{example}
\begin{figure}[ht]
\begin{center}
\includegraphics[scale=0.25]{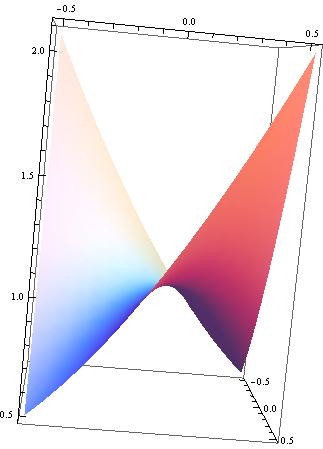}
\caption{A (linear) Weingarten affine translation surface of Type 1.}
\end{center}
\end{figure}
\begin{figure}[ht]
\begin{center}
\includegraphics[scale=0.25]{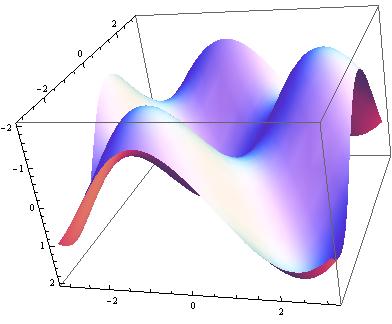}
\caption{An affine translation surface of Type 1 with $\bigtriangleup^{I}r_{i}=\lambda _{i}r_{i},$ $\left(\lambda _{1},\lambda_{2},\lambda_{3}\right)=\left(0,0,2\right)$.}
\end{center}
\end{figure}
\begin{figure}[ht]
\begin{center}
\includegraphics[scale=0.25]{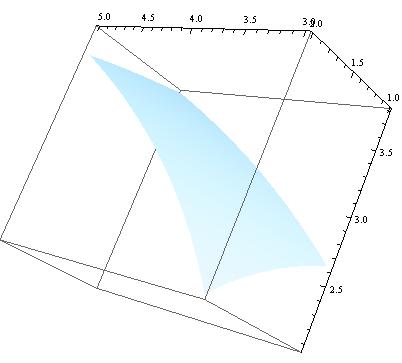}
\caption{An affine translation surface of Type 1 with $\bigtriangleup ^{II}r_{i}=\lambda
_{i}r_{i},$ $\left( \lambda _{1},\lambda _{2},\lambda _{3}\right) =\left(
1,1,0\right) .$}
\end{center}
\end{figure}

\end{document}